\documentclass[11pt]{article}
\usepackage[a4paper]{geometry}
\usepackage{amsthm}
\usepackage{amsmath}
\usepackage{amssymb}
\usepackage{amsfonts}
\usepackage{graphicx}
\usepackage{color}
\usepackage[bf,SL,BF]{subfigure}
\usepackage{url}
\usepackage{epstopdf}
\usepackage{pdfsync}
\usepackage[colorlinks]{hyperref}
\hypersetup{
    linkcolor=blue,
}

\usepackage{epsfig,amsbsy,graphicx,multirow}

\usepackage[all]{xy}

 \numberwithin{equation}{section}

\def\R{\mathbb{R}}

\def\E{\mathbb{E}}

\def\B{\mathcal{B}}

\def\L{\mathcal{L}}

\def\H{\mathcal{H}}

\def\<{\big\langle}
\def\>{\big\rangle}

\def\diiv{\operatorname{div}}

\definecolor{red}{rgb}{0.9, 0, 0}

\newtheorem{Theorem}{Theorem}[section]
\newtheorem{Proposition}[Theorem]{Proposition}
\newtheorem{Lemma}[Theorem]{Lemma}

\newtheorem{Remark}[Theorem]{Remark}
\newtheorem{Example}{Example}[section]

\begin{document}
\title{Bayesian Numerical Homogenization}

\date{\today}

\author{Houman Owhadi\footnote{California Institute of Technology, Computing \& Mathematical Sciences , MC 9-94 Pasadena, CA 91125, owhadi@caltech.edu}}

\maketitle

\begin{abstract}
Numerical homogenization, i.e. the finite-dimensional approximation of solution spaces of PDEs with arbitrary rough coefficients, requires the identification of accurate basis elements. These basis elements are oftentimes found after a laborious process of scientific investigation and plain guesswork. Can this identification problem be facilitated? Is there a general recipe/decision framework for guiding the  design of basis elements?
We suggest that the answer to the above questions could be positive based on the reformulation of numerical homogenization as a Bayesian Inference problem in which a given PDE with rough coefficients (or multi-scale operator) is excited with noise (random right hand side/source term) and one tries to estimate the value of the solution at a given point based on a finite number of observations. We apply this reformulation to the identification of bases for the numerical homogenization of arbitrary integro-differential equations and show that these bases have optimal recovery properties.
In particular we show how Rough Polyharmonic Splines can be re-discovered as the optimal solution of a Gaussian filtering problem.
\end{abstract}

%\maketitle
%\tableofcontents

\section{Bayesian Numerical Analysis}

This paper is inspired by a curious (and, perhaps, overlooked) link between Bayesian Inference and Numerical Analysis \cite{Diaconis:1988}, known as Bayesian Numerical Analysis \cite{Diaconis:1988, Shaw:1988, Hagan:1991, Hagan:1992}, and that can be traced back to Poincar\'{e}'s course on Probability Theory \cite{Poincare:1896}. We will recall Diaconis' compelling example \cite{Diaconis:1988} as an illustration of this link.

Let $f:[0,1]\rightarrow \R$ be a given function and assume that we are interested in the numerical approximation of $\int_0^1 f(t)\,dt$. The Bayesian approach to this quadrature problem  is to (1) Put a prior (probability distribution) on continuous functions $\mathcal{C}[0,1]$ (2) Calculate $f$ at $x_1,x_2,\ldots,x_n$ (to obtain the data $(f(x_1),\ldots,f(x_n))$) (3) Compute a posterior (4) Estimate $\int_0^1 f(t)\,dt$ by the Bayes rule.

If the prior on $\mathcal{C}[0,1]$ is that of a Brownian Motion (i.e. $f(t)=B_t$ where $B_t$ is a Brownian motion and $B_0$ is normal), then $\E\big[f(x)\big|f(x_1),\ldots,f(x_n)\big]$ is the piecewise linear interpolation of $f$ between the points $x_1,\ldots,x_n$ and one re-discovers the trapezoidal quadrature rule.

If the prior on $\mathcal{C}[0,1]$ is that of the first integral of a Brownian Motion (i.e. $f(t)\sim \int_0^t B_s\,ds$) then the posterior $\E\big[f(x)\big|f(x_1),\ldots,f(x_n)\big]$ is the cubic spline interpolant and integrating $k$ times yields splines of order $2k+1$.
Although this link has lead to the identification of new quadrature rules for numerical integration \cite{Hagan:1991}, it appears to have remained little known and our paper is prompted by the question of the existence of a similar link between Bayesian Inference and Numerical Homogenization.

As a prototypical example, consider the numerical homogenization of the PDE
\begin{equation}\label{eqn:scalar}
\begin{cases}
    -\diiv \Big(a(x)  \nabla u(x)\Big)=g(x) \quad  x \in \Omega,\, g \in L^2(\Omega),\\
    u=0 \quad \text{on}\quad \partial \Omega,
    \end{cases}
\end{equation}
where $\Omega$ is a bounded subset of $\R^d$ with piecewise Lipschitz boundary, $a$ is a symmetric, uniformly elliptic $d\times d$ matrix on $\Omega$ and with entries in $L^\infty(\Omega)$.

Recall that numerical homogenization concerns the approximation of the solution space of \eqref{eqn:scalar} with a finite-dimensional space.
Although classical homogenization concepts \cite{BeLiPa78, Mur78, Spagnolo:1968,Gio75, PapanicolaouVaradhan:1981, Kozlov:1979} might be present in some instances of this problem \cite{HoWu:1997, HouWu:1999, AllBri05, EEngquist:2003, Abdulle:2004, AnGlo06, Blanc:2006, Blanc:2007, EnSou08, BalJing10, Engquist:2011, Abdulle:2011}, one of the main objectives of numerical homogenization is to achieve a numerical approximation of the solution space of \eqref{eqn:scalar} with arbitrary rough coefficients, i.e., in particular, without the assumptions found in classical homogenization, such as scale separation, ergodicity at fine scales and $\epsilon$-sequences of operators.
In this situation,  piecewise linear finite-elements  can perform arbitrarily badly  \cite{BaOs:2000} and the numerical approximation of the solution
space involves the identification of accurate basis elements adapted to the microstructure $a(x)$
 \cite{WhHo87, BaOs:1983, BaCaOs:1994, OwZh:2007a, OwZh06c, OwZh:2007b, EfHo:2007, EfGiHouEw:2006, Arbogast:2007, Arbogast:2006, NoPaPi:2008, Caffarelli:2008, BraWu09, BraWu09, ChuHou09, BeOw:2010, BaLip10, OwZh:2011, EfGaWu:2011, MaPe:2012, GrGrSa2012}.

As for the identification of quadrature rules in numerical analysis, the identification of accurate basis elements in numerical homogenization has been based on a difficult process of scientific investigation.  Let us now turn our attention to the Bayesian approach to this problem. An immediate question is where do we place the prior? (1) If the prior is placed on $u$ then posterior values do not see (depend on) the microstructure. (2) If the prior is placed on $a$ then the microstructure becomes random whereas our purpose is the numerical homogenization of a given deterministic microstructure. Let us also note that the randomization of the microstructure, as investigated by Polynomial Chaos Approximation/Stochastic Expansion methods
 \cite{GhanemDham:1998, Ghanem:1999, Xiu:2009, BabuskaNobRa:2010, EldredWebsterConstantine:2008,  DoostanOwhadi:2010}, does not lead to the simplification seen after homogenization but to increased complexity with the  dimension of input stochastic  variables \cite{Todor07a, Bieri09a} (although Stochastic Expansion methods have been used successfully to beat  Monte-Carlo sampling they do not lead to averaging results seen in homogenization). (3) If the prior is placed on $g$ then the noise propagates through the microstructure and the posterior value of $u$ contains that information.

This observation motivates us to place the prior on the source term $g$ in \eqref{eqn:scalar}, e.g., replace it by white noise (i.e. a centered Gaussian field $\xi(x)$ on $\Omega$ with covariance function $\delta(x-y)$) and consider the stochastic PDE
\begin{equation}\label{eqn:scalarstochastic}
\begin{cases}
    -\diiv \Big(a(x)  \nabla u(x)\Big)=\xi(x) \quad  x \in \Omega,\\
    u=0 \quad \text{on}\quad \partial \Omega.
    \end{cases}
\end{equation}
 Observe that the solution \eqref{eqn:scalarstochastic} at the point $x$, $u(x)$, is a random variable and its best (mean squared) approximation given $u(x_1),\ldots,u(x_N)$ (the  values of the solution of \eqref{eqn:scalarstochastic} at the points $x_1,\ldots,x_N$ form the data) is its conditional expectation $\E\big[u(x)\big|u(x_1),\ldots,u(x_N)\big]$. One result of this paper is that
 \begin{equation}\label{eq:meangaussprps}
 \E\big[u(x)\big|u(x_1),\ldots,u(x_N)\big]=\sum_{i=1}^N u(x_i)\phi_i(x),
 \end{equation}
 where the functions $\phi_i$ are Rough Polyharmonic Splines \cite{OwhadiZhangBerlyand:2014} (RPS) which have been identified as accurate basis elements for the numerical homogenization of \eqref{eqn:scalar}  having noteworthy variational, optimal recovery and localization properties. The discovery of these Rough Polyharmonic Splines has required a significant amount of work and trial and errors but here, they are identified after a single step of Bayesian conditioning.

This observation motivates us to investigate what the same process of Bayesian conditioning would give under  different priors and under other observations  than the values of $u$ at individual points (we will consider data formed by the values of a finite number of linear functions of $u$). In particular, we will use this link between Bayesian Inference and Numerical Homogenization to identify bases for the numerical homogenization of arbitrary linear integro-differential equations. Our purpose is to show that this link is generic and could in principle be used, beyond numerical homogenization, as a guiding principle for the coarse-graining of multi-scale systems. The Bayesian approach to this problem  is to (1) Put a prior on the degrees of freedom of the system (2) Select a finite number of coarse variables (3) Compute the posterior value of the state of the system conditioned on the coarse variables.

\section{General setup}

Let $\L$ and $\B$ be linear integro-differential operators on $\Omega$ and $\partial \Omega$ such that
(1) $(\L,\B): \H(\Omega)\rightarrow \H_{\L}(\Omega)\times \H_{\B}(\partial \Omega)$, where $\H(\Omega)$, $\H_{\L}(\Omega)$ and $\H_{\B}(\partial \Omega)$ are Hilbert spaces of Generalized functions on $\Omega$ and $\partial \Omega$ (2)  $\H_{\L}(\Omega)$ contains $L^2(\Omega)$ and $\H(\Omega)$ is contained in $L^2(\Omega)$.

Consider the integro-differential equation
\begin{equation}\label{eqn:scalargeneral}
\begin{cases}
    \L u(x)=g(x) \quad  x \in \Omega,\\
    \B u=0 \quad \text{on}\quad \partial \Omega.
    \end{cases}
\end{equation}
As with \eqref{eqn:scalar} the numerical homogenization of \eqref{eqn:scalargeneral} will require the assumption that $g$ belongs to a strict subspace of $\H_{\L}(\Omega)$.

We will assume that $\L$ and $\B$ are such that \eqref{eqn:scalargeneral} (1) admits a unique solution in  $\H(\Omega)$ (2) and a Green's function
$G$.  Recall that $G$ is defined as the solution of
\begin{equation}\label{eqn:scalargreen}
\begin{cases}
    \L G(x,y)=\delta(x-y) \quad  x \in \Omega,\\
    \B G(x,y)=0 \quad \text{for}\quad x\in \partial \Omega,
    \end{cases}
\end{equation}
where $\delta(\cdot-y)$ is the Delta mass of dirac at the point $y$.

\begin{Example}\label{eg:8gy2ys}
Note that for the prototypical example \eqref{eqn:scalar} we have
\begin{equation}
\L u(x):=-\diiv\big(a(x)\nabla u(x)\big)\text{ and }\B u(x)=u(x).
\end{equation}
\end{Example}

Our purpose is to identify a good basis for the numerical homogenization or coarse-graining of \eqref{eqn:scalargeneral}.

\section{Bayesian Numerical Homogenization}
Our Bayesian approach to the numerical homogenization of \eqref{eqn:scalargeneral} is to replace the source term $g$ by a Gaussian field $\xi$.
More precisely we introduce $\xi$, a centered Gaussian field on $\Omega$ with covariance function
\begin{equation}\label{eqlambda}
\Lambda(x,y):=\E\big[\xi(x)\xi(y)\big],
\end{equation}
and consider the stochastic  integro-differential equation

\begin{equation}\label{eqn:noisy}
\begin{cases}
    \L u(x)=\xi(x) \quad  x \in \Omega,\,\\
    \B u=0 \quad \text{on}\quad \partial \Omega.
    \end{cases}
\end{equation}

\begin{Proposition}
The solution of \eqref{eqn:noisy} is a Gaussian field on $\Omega$ whose covariance function $\Gamma(x,y):=\E\big[u(x)u(y)\big]$ is
\begin{equation}\label{eq:dergamma}
\Gamma(x,y)=\int_{\Omega^2}G(x,z) \Lambda(z,z') G(y,z')\,dz\,dz'.
\end{equation}
\end{Proposition}
\begin{Remark}\label{rmllstar}
Write $(\L^*,\B^*)$ the adjoint of $(\L,\B)$ with respect to the (scalar) product defined on $\H(\Omega)$ by
 $\<u,v\>_{L^2}:=\int_\Omega  u(x) v(x)\,dx$. Observe that $G(y,x)$ (the transpose of $G(x,y)$ with respect to the scalar product $\<\cdot,\cdot\>_{L^2}$)
is the Green's function of $(\L^*,\B^*)$ (the complex conjugation of the
Green's function is not required to define its adjoint because the scalar product is bilinear and not sesquilinear).
Observe that if $\xi$ is white noise (i.e. $\Lambda(x-y)=\delta(x-y)$) then
\begin{equation}\label{eq:dergammawhitenoise}
\Gamma(x,y)=\int_{\Omega}G(x,z)  G(y,z)\,dz,
\end{equation}
which is the Kernel of $\L^* \L$, i.e., $\L^* \L \Gamma(x,y)=\delta(x-y)$.
\end{Remark}
\begin{proof}
Since $\L$ and $\B$ are linear operators, $u$ is a linear function of $\xi$ and is therefore a Gaussian field. Moreover its covariance function is given by
\begin{equation}
\begin{split}
\Gamma(x,y)&=\E\big[u(x)u(y)\big]=\E\big[\int_{\Omega^2} G(x,z) \xi(z) G(y,z') \xi(z')\big]\,dz\,dz'\\
&=\int_{\Omega^2} G(x,z)  G(y,z') \E\big[\xi(z)\xi(z')\big]\,dz\,dz',
\end{split}
\end{equation}
which finishes the proof.
\end{proof}

\begin{Remark}
Beyond Bayesian Homogenization, equations with random right hand side can also  be of interest in practical applications, for instance in the modeling of the electrostatics in nanoscale field-effect sensors, where  fluctuations arise from random charge concentrations \cite{Heitzinger:2014}.
\end{Remark}

\subsection{On the choice of the noise}

We will show that the choice of the noise $\Lambda$ can be determined by the regularity of the source term $g$ in the right hand side of \eqref{eqn:scalargeneral}.
More precisely if $\xi$ is white noise ($\Lambda(x,y)=\delta(x-y)$) then the resulting accuracy estimates will be obtained under the assumption that $g\in L^2(\Omega)$ and as a function of $\|g\|_{L^2(\Omega)}$.

If $\xi$ is not white noise (i.e. if its covariance function is not $\delta(x-y)$) then we assume that there exists two linear integro-differential operators
$\L_{\Lambda}$ and $\B_{\Lambda}$ such that $\xi$ is the stochastic solution of the following equation with white noise $\xi'$ as the source term:
\begin{equation}\label{eqn:noisyxi}
\begin{cases}
    \L_{\Lambda} \xi(x)=\xi'(x) \quad  x \in \Omega,\,\\
    \B_{\Lambda} \xi=0 \quad \text{on}\quad \partial \Omega.
    \end{cases}
\end{equation}

In what follows, if $\xi$ is not white noise then we assume it to be obtained as in \eqref{eqn:noisyxi} and the resulting accuracy estimates will be obtained under the assumption that $\L_\Lambda g\in L^2(\Omega)$ and as a function of $\|\L_\Lambda g\|_{L^2(\Omega)}$.
A prototypical example corresponds to the situation where $\xi$ is obtained as the regularization of white noise via a power of the Laplace Dirichlet operator on $\Omega$ and this allows us to identify optimal recovery bases under the assumption that $g\in H^s(\Omega)$ with $s\geq 0$ or $s<0$.

\subsection{Identification of basis elements via conditioning}
Let $N$ be a strictly positive integer. Our Bayesian approach is based on the conditioning of the solution of \eqref{eqn:noisy} posterior to the observation of $N$ linear functions of $u(x)$, expressed as
\begin{equation}
\int_{\Omega}u(x)\psi_i(x)\,dx\quad i\in \{1,\ldots,N\},
\end{equation}
where $\psi_1,\ldots,\psi_N$ are $N$ linearly independent generalized functions (distributions) on $\Omega$ such that for all $i$
\begin{equation}\label{eq:gentestfunccond}
\int_{\Omega^2} \psi_i(x) \Gamma(x,y) \psi_i(y) \,dx\,dy<\infty.
\end{equation}
Examples of $\psi_i$ include masses of Dirac ($\psi_i(x)=\delta(x-x_i)$), indicator functions of subsets of $\Omega$ and elements of $L^1(\Omega)$.
Let $\Theta$ be the $N\times N$ symmetric matrix defined by
\begin{equation}\label{eq:gentestfuncthet}
\Theta_{i,j}:=\int_{\Omega^2} \psi_i(x) \Gamma(x,y) \psi_j(y) \,dx\,dy.
\end{equation}
Note that \eqref{eq:gentestfunccond} implies that if $u$ is the solution of \eqref{eqn:noisy} then
 \begin{equation}\label{eq:Psidef}
 \Psi:=\big(\int_{\Omega}u(x) \psi_1(x)\,dx,\ldots, \int_{\Omega}u(x) \psi_N(x)\,dx\big),
 \end{equation}
  is a well defined center Gaussian random vector with covariance matrix $\Theta$.

We will from now on assume that the covariance function \eqref{eqlambda} is not degenerate in the sense that for $f\in \H(\Omega)$,
\begin{equation}
\|f\|_\Lambda^2:=\int_{\Omega} f(x) \Lambda(x,y) f(y)\,dx\,dy
 \end{equation}
 is zero if and only if $f$ is the null function. Note that if $\xi$ is obtained via \eqref{eqn:noisyxi} then $\|f\|_{\Lambda}^2=\|\L_\Lambda^{-1} f\|_{L^2(\Omega)}^2$ (writing  $\L_\Lambda^{-1} f$ the solution of $\L_{\Lambda} u=f$ in $\Omega$ with $\B_{\Lambda} u=0$ on $\partial \Omega$)  and the non-degeneracy of $\Lambda$ is equivalent to that of  the operator $\L_\Lambda$.

\begin{Lemma}\label{lem:Theta}
The $N\times N$ matrix $\Theta$ is symmetric positive  definite. Furthermore for all $l\in \R^{N}$,
\begin{equation}
l^T \Theta l=\|v\|_{\Lambda}^2,
\end{equation}
where $v$ is the solution of
\begin{equation}\label{eqn:scalarGreensTheta}
\begin{cases}
    \L^* v(x)=\sum_{j=1}^{N} l_j \psi_j(x)& \text{ for }  x \in \Omega,\\
   \B^* v(x)=0 & \text{ for }x\in \partial \Omega.
    \end{cases}
\end{equation}
\end{Lemma}
\begin{proof}
We obtain from  \eqref{eq:dergamma} that for $l\in \R^N$
\begin{equation}
l^T \Theta l=\int_{\Omega^2} (\int_{\Omega}\sum_{i=1}^N \psi_i(x) G(x,z)\,dx) \Lambda(z,z') (\int_{\Omega}\sum_{j=1}^N \psi_j(y) G(y,z')\,dy)\,dz\,dz'.
\end{equation}
Write
\begin{equation}
v(x):=\sum_{i=1}^N  l_i \int_{\Omega}G(y,x)\psi_i(y)\,dy.
\end{equation}
Since $G(\cdot,x)$ is the Green's function of the adjoint operator (Remark \ref{rmllstar}) it follows that $v$ is the solution of \eqref{eqn:scalarGreensTheta} and $\|v\|_{\Lambda}^2=l^T \Theta l$ which implies that $\Theta$ is symmetric positive definite. Indeed if $\Theta$ is not positive definite, then there would exist a non zero vector $l\in \R^N$ such that $\Theta l=0$. This would imply $\|v\|_{\Lambda}=0$ which is a contradiction since the equation \eqref{eqn:scalarGreensTheta} has a non zero solution (since $l\not=0$ and the $\psi_i$ are linearly independent).
\end{proof}

Our motivation for using Gaussian noise in \eqref{eqn:noisy} lies in the fact that for Gaussian fields, conditional expected values can be computed via linear projection. Henceforth our approach is also akin to Gaussian filtering for numerical homogenization and the following Theorem shows that this approach allows
for the identification of a (projection) basis $\phi_i$.
\begin{Theorem}\label{thm:bayescondnumhom}
Let $u$ be the solution of \eqref{eqn:noisy} and $\Psi$ defined by \eqref{eq:Psidef}, then
\begin{equation}\label{eq:condexpphi}
\E\big[u(x)\big|\Psi\big]=\sum_{i=1}^N  \Psi_i \phi_{i}(x),
\end{equation}
with
\begin{equation}\label{eq:defpsii}
\Psi_i:=\int_{\Omega}u(y) \psi_i(y)\,dy,
\end{equation}
and
\begin{equation}\label{eq:phiidef}
\phi_i(x):=\sum_{j=1}^N \Theta^{-1}_{i,j} \int_{\Omega} \Gamma(x,y)\psi_j(y)\,dy.
\end{equation}
Furthermore, $u(x)$ conditioned on  the value of  $\Psi$ is a Gaussian random variable with mean \eqref{eq:condexpphi} and variance
\begin{equation}\label{eq:ggjhgs723}
\sigma(x)^2=\Gamma(x,x)-\sum_{i,j=1}^N \Theta^{-1}_{i,j}\int_{\Omega} \Gamma(x,y)\psi_j(y)\,dy \int_{\Omega} \Gamma(x,y)\psi_i(y)\,dy.
\end{equation}
\end{Theorem}
\begin{proof}
Let
\begin{equation}
u_\Psi(x):=\E\big[u(x)\big|\Psi\big].
\end{equation}
Since $u$ and $\Psi$ belong to the same Gaussian space, it follows that $u_\Psi$ is a linear function of $\Psi$ obtained by minimizing the mean squared error
\begin{equation}\label{eq:gyg2ugw}
\E\Big[\big(u(x)-c\cdot \Psi\big)^2\Big]=\Gamma(x,x)-2\sum_{i=1}^N c_i \int_{\Omega} \Gamma(x,y)\psi_i(y)\,dy+\sum_{i,j=1}^N c_i c_j \Theta_{i,j},
\end{equation}
with respect to $c\in \R^N$, where $\Theta$ is defined by \eqref{eq:gentestfuncthet}. We conclude the proof by identifying the minimizer in $c$, using Lemma \ref{lem:Theta} for the invertibility of $\Theta$ and noting that \eqref{eq:ggjhgs723} is simply \eqref{eq:gyg2ugw} at the minimum in $c$.
\end{proof}

\begin{Example}
If $\L$ and $\B$ correspond to the prototypical example \eqref{eqn:scalar} (see also Example \ref{eg:8gy2ys}), if $\xi$ is white noise (i.e. if its covariance matrix is $\Lambda(x,y)=\delta(x-y)$), and if the observable functions are masses of Diracs at points $x_i\in \Omega$ (and $d\leq 3$ which is required for \eqref{eq:gentestfunccond}), then
Theorem \ref{thm:bayescondnumhom} implies \eqref{eq:meangaussprps} and the basis elements $\phi_i$ are the RPS elements of \cite{OwhadiZhangBerlyand:2014}
which are a generalization of Polyharmonic Splines to PDEs with rough coefficients. Recall that
Polyharmonic splines can be traced back to the seminal work of Harder and Desmarais \cite{Harder:1972} and Duchon  \cite{Duchon:1976,Duchon:1977,Duchon:1978}.

Note also that according to Theorem \eqref{thm:bayescondnumhom} the process of Bayesian conditioning gives us the whole posterior distribution of $u(x)$ and not only its (conditional) expected value. In particular, the distribution of $u(x)$ conditioned on   $u(x_1),\ldots,u(x_N)$ is a Gaussian random variable with mean \eqref{eq:meangaussprps} and variance
\begin{equation}\label{eq:meangaussprdeps}
\sigma^2(x)= \Gamma(x,x)-\sum_{i,j=1}^N \Theta^{-1}_{i,j} \Gamma(x,x_j) \Gamma(x,x_i),
\end{equation}
and this observation can be used to compute the probability of deviation of the RPS interpolation from $u(x)$ by a given margin and guide the addition of interpolation points (note that $\sigma^2(x)=0$ at the interpolation points $x_1,\ldots,x_N$).
\end{Example}
\begin{Remark}
We will show in Theorem \ref{thm:gammaproof} that $\sigma(x)$ also controls the pointwise  error between  the  solution of the original integro-differential equation \eqref{eqn:scalargeneral} and the approximation $\sum_{i=1}^N \phi_i(x) \int_{\Omega}u(y)\psi_i(y)\,dy$.
\end{Remark}

\section{Variational properties of basis elements}\label{subsec:whitenoise}

In this section we will show that as for RPS \cite{OwhadiZhangBerlyand:2014}, the basis elements $\phi_i$ from Bayesian Inference have remarkable variational and optimal recovery properties that  can be used (1) for their practical computation (2) for the derivation of accuracy estimates.

\subsection{White Gaussian noise}
In this subsection we will assume that $\xi$ is  white noise (i.e. $\Lambda(x,y)=\delta(x-y)$). Define
\begin{equation}
\begin{split}
V:=\big\{\phi \in \H(\Omega) \big| \L \phi\in L^2(\Omega)\text{ and } \B\phi=0\text{ on }  \partial \Omega\big\},
\end{split}
\end{equation}
and let $\<\cdot,\cdot\>$ be the (scalar) product on $V$ defined by: for $u, v\in V$,
\begin{equation}
 \<u, v\> := \int_\Omega \big(\L u(x)\big)\big(\L v(x)\big)\,dx.
\end{equation}
Note in particular that $\<v,v\>=0$ if and only if $v=0$ and we write
\begin{equation}\label{eq:normV}
 \|v\|_V:=\<v, v\>^\frac{1}{2},
\end{equation}
the corresponding norm (note that $\|v\|_V$ is a norm on $V$ because $\|v\|_V=0$ and $v\in V$ imply $\L v=0$ in $\Omega$ and $\B v=0$ on $\partial \Omega$ which leads to $v=0$ by the non-degeneracy of the operator $\L$).
\begin{Theorem}\label{thm:rkhs}
If $\Gamma(x,x)<\infty$ then for $v\in V$ and $x\in \Omega$
\begin{equation}
\big|v(x)\big| \leq \big(\Gamma(x,x)\big)^\frac{1}{2} \|v\|_V,
\end{equation}
and the space $V$ with
the reproducing Kernel $\Gamma(x,y)$ forms a  Reproducing Kernel Hilbert Space. In particular, for all $v\in V$
\begin{equation}
\<v,\Gamma(\cdot,x)\>=v(x).
\end{equation}
\end{Theorem}
\begin{proof}
Theorem \ref{thm:rkhs} is a direct consequence of the fact that
\begin{equation}
\<v,\int_{\Omega}\Gamma(\cdot,y)f(y)dy\>=\int_{\Omega} v(y)f(y)\,dy,
\end{equation}
and (by Cauchy-Schwartz inequality and $\<\Gamma(\cdot,x),\int_{\Omega}\Gamma(\cdot,x)\>=\Gamma(x,x)$)
\begin{equation}
\<v,\Gamma(\cdot,x)\>\leq \<v,v\>^\frac{1}{2} \big(\Gamma(x,x)\big)^\frac{1}{2}.
\end{equation}
\end{proof}
Define
\begin{equation}
\begin{split}
V_i:=&\big\{\phi \in V \big| \int_\Omega \phi(x)\psi_i(x)\,dx=1\text{ and }\int_\Omega \phi(x)\psi_j(x)\,dx=0\\&\text{ for }j\in \{1,\ldots,N\} \text{ such that } j\not=i\big\},
\end{split}
\end{equation}
and consider
the following optimization problem over $V_i$:
\begin{equation}\label{eqn:convexopt}
\begin{cases}
\text{Minimize }  \<\phi, \phi\> \\
\text{Subject to }  \phi\in V_i.
\end{cases}
\end{equation}

\begin{Proposition}\label{prop:nonempty}
$V_i$ is a non-empty closed affine subspace of $V$.
Problem \eqref{eqn:convexopt} is a strictly convex quadratic optimization problem over $V_i$. The unique minimizer of \eqref{eqn:convexopt} is  $\phi_i$ as defined by
\eqref{eq:phiidef}.
\end{Proposition}

\begin{proof}
Let us first prove that $\phi_i \in V_i$. Let
\begin{equation}\label{eq:defthetai}
\theta_i(x):=\int_{\Omega}\Gamma(x,y)\psi_i(y)\,dy.
\end{equation}
First observe that for all $i \in \{1,\ldots,N\}$,
\begin{equation}\label{eq:ugwygwugye}
\L \theta_i(x)= \int_{\Omega} G(y,x)\psi_i(y)\,dy,
\end{equation}
and $\B\theta_i(x)=0$ on $\partial \Omega$.
 Noting that $\big\|\L \theta_i \big\|_{L^2(\Omega)}^2=\Theta_{i,i}$ we deduce from \eqref{eq:gentestfunccond}
  that $\theta_i \in V$. We conclude from \eqref{eq:phiidef} and Lemma \ref{lem:Theta} that $\phi_i \in V$.
Now observe that \eqref{eq:gentestfuncthet} implies  that
\begin{equation}
\int_{\Omega}\phi_i(x)\psi_j(x)= (\Theta^{-1} \cdot \Theta)_{i,j}= \delta_{i,j},
\end{equation}
where $\delta_{i,i}=1$ and $\delta_{i,j}=0$ for $j\not=i$.
We conclude that $\phi_i \in V_i$ which implies that   $V_i$ is non empty (it is  easy to check that it is a closed affine sub-space of $V$).

Now let us prove that problem \eqref{eqn:convexopt} is a strictly convex optimization problem over $V_i$. Let $v, w \in V_i$ such that $v\not=w$. Write for $\lambda \in [0,1]$,
\begin{equation}\label{eqn:convseept}
f(\lambda):=  \<v+\lambda (w-v),v+\lambda (w-v) \>,
\end{equation}
and we need to show that $f(\lambda)$ is a strictly convex function.
Observing that
\begin{equation}\label{eqn:convseepshiuhiwt}
f(\lambda)= \<v,v\>+ 2\lambda \<v,w-v\> +\lambda^2 \<v-w,v-w\>,
\end{equation}
and noting that $\<v-w,v-w\> >0$ (otherwise one would have $v=w$) we deduce that $f$ is strictly convex in $\lambda$. We conclude that (see, for example, \cite[pp. 35, Proposition 1.2]{EkTe:1987})
that Problem \eqref{eqn:convexopt} is a strictly convex optimization problem over $V_i$ and that it admits a unique minimizer in $V_i$.
We will postpone the proof of the fact that $\phi_i$ is the minimizer of \eqref{eqn:convexopt} to  the proof of Theorem \ref{lem:minimizingproperty}.
\end{proof}

\begin{Remark}
 It is important to note that in practical (numerical) applications each element
$\phi_i$ would be obtained by solving the quadratic optimization problem \eqref{eqn:convexopt} rather than through the representation formula \eqref{eq:phiidef} because the identification of $\Gamma$ in \eqref{eq:phiidef} is more expensive than solving the linear systems associated with \eqref{eqn:convexopt} (inverting a matrix is more expensive than solving a linear system).
 Note also that, if $u$ is the (stochastic) solution of \eqref{eqn:noisy}, then $\phi_i$ is also equal to the expected value of $u(x)$ conditioned on $\int_{\Omega} u(x)\psi_i(x)=1$ and $\int_{\Omega} u(x)\psi_j(x)=0$ for $j\not=i$, i.e.
\begin{equation}
\phi_i(x)=\E\big[u(x)\big|\text{$\int_{\Omega} u(x)\psi_i(x)=1$ and $\int_{\Omega} u(x)\psi_j(x)=0$ for $j\not=i$} \big].
\end{equation}
\end{Remark}

\begin{Remark}
A simple calculation allows us to show that $\phi_i$ is also the solution of the following nested equations
\begin{equation}\label{eqn:nested1}
\begin{cases}
    \L \phi_i(x)=\chi_i(x) \quad  x \in \Omega,\,\\
    \B \phi_i=0 \quad \text{on}\quad \partial \Omega,
    \end{cases}
\end{equation}
\begin{equation}\label{eqn:nested1bis}
\begin{cases}
    \L^* \chi_i(x)=\sum_{j=1}^N \Theta^{-1}_{i,j}\psi_j(x) \quad  x \in \Omega,\,\\
    \B^* \chi_i(x)=0 \quad \text{on}\quad \partial \Omega.
    \end{cases}
\end{equation}
\end{Remark}

\begin{Remark}
Another simple calculation allows us to show that $\phi_i$ is also the solution of the following nested equations
\begin{equation}\label{eqn:nested1g}
\begin{cases}
    \L \phi_i(x)=\chi_i(x) \quad  x \in \Omega,\,\\
    \B \phi_i=0 \quad \text{on}\quad \partial \Omega,\\
    \int_{\Omega} \phi_i(x) \psi_j(x)\,dx=\delta_{i,j} \text{ for }j\in\{1,\ldots,N\},
    \end{cases}
\end{equation}
\begin{equation}\label{eqn:nested1gg}
\begin{cases}
    \L^* \chi_i(x)=\sum_{j=1}^N c_j \psi_j(x) \quad  x \in \Omega,\,\\
    \B^* \chi_i(x)=0 \quad \text{on}\quad \partial \Omega,
    \end{cases}
\end{equation}
where $c\in \R^N$ is an unknown vector determined by the third equation in \eqref{eqn:nested1g}.
\end{Remark}

Write  $V_0$ the subset of $V$ defined by
\begin{equation}\label{eq:hiuhiu33}
 V_0:=\big\{v\in V: \int_{\Omega}v(x)\psi_i(x)\,dx=0, \forall i\in \{1,\ldots,N\}\big\}.
\end{equation}

\begin{Theorem}\label{lem:minimizingproperty}
It holds true that
\begin{itemize}
\item The basis $\phi_i$ is orthorgonal to $V_0$ with respect to the product $\<\cdot,\cdot\>$, i.e.
\begin{equation}\label{eqn:orthogonality}
 \<\phi_i,v\> =0, \quad \forall i\in \{1,\ldots,N\} \text{ and } \forall v\in V_0.
\end{equation}
\item $\sum_{i=1}^N w_i \phi_i$ is the unique minimizer of $\<v,v\>$
over all $v\in V$ such that\\ $\int_{\Omega}v(x)\psi_i(x)\,dx=w_i$.
\item For all $i\in \{1,\ldots,N\}$ and for all $v\in V$,
\begin{equation}\label{eqn:orthogonalidoddoty}
 \<\phi_i,v\> =\sum_{j=1}^{N} \Theta^{-1}_{i,j}\, \int_{\Omega}v(x)\psi_j(x)\,dx.
\end{equation}
\item For all $i,j\in \{1,\ldots,N\}$,
\begin{equation}\label{eqn:orthogjuoddoty}
 \<\phi_i,\phi_j\> =\Theta^{-1}_{i,j}.
\end{equation}
\end{itemize}
\end{Theorem}
\begin{Remark}
Theorem \ref{lem:minimizingproperty} and its proof is analogous to the optimal property of strictly conditionally positive definite kernels  \cite{Wendland:2005} when used as interpolant solutions of the optimal recovery problem \cite{GolombWeinberger:1959}.
\end{Remark}

\begin{proof}
 We have, using \eqref{eq:defthetai}, \eqref{eq:phiidef} and \eqref{eq:ugwygwugye}
\begin{equation}
 \<\phi_i,v\> =\sum_{j=1}^N \Theta^{-1}_{i,j} \int_{\Omega} \L \theta_j(x) \L v(x)\,dx=\sum_{j=1}^N \Theta^{-1}_{i,j} \int_{\Omega} \psi_j(y) v(y)\,dy=0,
\end{equation}
Which implies \eqref{eqn:orthogonality}, \eqref{eqn:orthogonalidoddoty} and \eqref{eqn:orthogjuoddoty}.

Let $w\in \R^N$ and  $\phi_w:=\sum_{i=1}^N w_i \phi_i$. Let $v\in V$ such that $\int_{\Omega}v(x)\psi_i(x)\,dx=w_i$ for all $i\in \{1,\ldots,N\}$. Since $\phi_w-v \in V_0$, it follows that
\begin{equation}
 \<v,v\> = \<\phi_w,\phi_w\> +\<v-\phi_w,v-\phi_w\>.
\end{equation}
It follows that $\sum_{i=1}^N w_i \phi_i$ is the unique minimizer of $\<v,v\>$
over all $v\in V$ such that $\int_{\Omega}v(x)\psi_i(x)\,dx=w_i$. Note that this also implies that
$\phi_i$ is the minimizer of \eqref{eqn:convexopt}.
\end{proof}
\subsection{Non-white Gaussian noise}
If $\xi$ is not white noise (i.e. $\Lambda(x,y)\not=\delta(x-y)$) then Theorem \ref{thm:rkhs},Theorem \ref{lem:minimizingproperty} and Proposition \ref{prop:nonempty} remain true provided that the definitions of the space $V$ and scalar product $\<\cdot,\cdot\>$ are changed to
\begin{equation}
\begin{split}
V:=\big\{\phi \in \H(\Omega) \big| \L_{\Lambda}\L \phi\in L^2(\Omega),\, \B\phi=0\text{ and }\B_{\Lambda}\L\phi=0  \text{ on }\partial \Omega\big\},
\end{split}
\end{equation}
\begin{equation}
 \<u, v\> := \int_\Omega \big(\L_{\Lambda}\L u(x)\big)\big(\L_{\Lambda}\L v(x)\big)\,dx,
\end{equation}
where  $\L_{\Lambda}$ and $\B_{\Lambda}$ are defined in \eqref{eqn:noisyxi}.

\section{Accuracy of the basis elements $\phi_i$}

\subsection{Pointwise estimates}
Let $\|v\|_V$ be defined as in \eqref{eq:normV}.

\begin{Theorem}\label{thm:gammaproof}
Assume that $\Gamma(x,x)<\infty$. Let $v\in V$. It holds true that for $x\in \Omega$
\begin{equation}
\Big|v(x)-\sum_{i=1}^N  \phi_i(x) \big(\int_{\Omega}v(y)\psi_i(y)\,dy \big) \Big| \leq \sigma(x) \|v\|_V,
\end{equation}
where $\sigma^2(x)$ is the variance of  $u(x)$ (solution of \eqref{eqn:noisy})   conditioned on \\  $\int_{\Omega}u(y)\psi_1(y)\,dy,\ldots,\int_{\Omega}u(y)\psi_N(y)\,dy$  as defined by \eqref{eq:ggjhgs723}.
In particular if $u$ is the solution of the original integro-differential equation \eqref{eqn:scalargeneral}, then
\begin{equation}
\Big|u(x)-\sum_{i=1}^N \phi_i(x) \big(\int_{\Omega}u(y)\psi_i(y)\,dy \big)\Big| \leq \sigma(x) \|g\|_{L^2(\Omega)},
\end{equation}
 if $\phi_i,\sigma$ are derived from white noise, and
\begin{equation}
\Big|u(x)-\sum_{i=1}^N \phi_i(x) \big( \int_{\Omega}u(y)\psi_i(y)\,dy\big)\Big| \leq \sigma(x) \|\L_\Lambda g\|_{L^2(\Omega)},
\end{equation}
 if $\phi_i,\sigma$ are derived from the noise with covariance function $\Lambda$ described in \eqref{eqn:noisyxi}.
\end{Theorem}
\begin{proof}
Let $v\in V$ and $x\in \Omega$.  Using the reproducing kernel property of Theorem \ref{thm:rkhs} we obtain that
\begin{equation}
\big|v(x)-\sum_{i=1}^N \phi_i(x) \int_{\Omega}v(y)\psi_i(y)\,dy\big| =\Big|\<v,\Gamma(\cdot,x)-\sum_{i=1}^N \phi_i(x)
\int_{\Omega}\Gamma(\cdot,y)\psi_i(y)\,dy\>\Big|.
\end{equation}
Therefore, using Cauchy-Schwartz inequality
\begin{equation}\label{eq:iuiue}
\big|v(x)-\sum_{i=1}^N \phi_i(x) \int_{\Omega}v(y)\psi_i(y)\,dy\big| \leq \|v\|_V \big\|\Gamma(\cdot,x)-\sum_{i=1}^N \phi_i(x)
\int_{\Omega}\Gamma(\cdot,y)\psi_i(y)\,dy\big\|_V.
\end{equation}
We conclude by expanding the right hand side of \eqref{eq:iuiue} and the definition $\phi_i(x)=\sum_{j=1}^N \Theta^{-1}_{i,j} \int_{\Omega}\Gamma(x,y)\psi_i(y)\,dy$.
\end{proof}

\begin{Remark}
$\sigma^2(x)$ is also known as the Power function in radial basis function interpolation \cite{Wendland:2005, Fasshauer:2005}.
The proof of Theorem \ref{thm:gammaproof} is similar to the one used to derive local error estimates for radial basis function interpolation
of scattered data (see  \cite{WuSchback:93} in which  $\sigma^2(x)$ was referred to as the Kriging function, a terminology coming from geostatistics \cite{Myers:1992}).
\end{Remark}

\subsection{$\H(\Omega)$-norm estimates}

Let  $V_0$ be the subset of $V$ defined by \eqref{eq:hiuhiu33}. Write
\begin{equation}
\rho(V_0):=\sup_{v\in V_0}\frac{\|v\|_{\H(\Omega)}}{\|v\|_V},
\end{equation}
where $\|.\|_{\H(\Omega)}$ is the natural norm associated with the space on which the operator $\L$ is defined.

\begin{Theorem}\label{thm:acchtov}
We have for all $v\in V$
\begin{equation}\label{eq:kshuh}
\Big\|v-\sum_{i=1}^N \phi_i \big(\int_{\Omega}v(y)\psi_i(y)\,dy\big)\Big\|_{\H(\Omega)}\leq  \rho(V_0) \|v\|_V,
\end{equation}
and $\rho(V_0)$ is the smallest constant for which \eqref{eq:kshuh} holds for all $v\in V$.
\end{Theorem}
\begin{proof}
Write $v_\Psi(x):=\sum_{i=1}^N \phi_i(x) \big(\int_{\Omega}v(y)\psi_i(y)\,dy\big)$

Observing that $v-v_\Psi$ belongs to $V_0$ implies that
\begin{equation}
\|v-v_\Psi\|_{\H(\Omega)}\leq  \rho(V_0) \<v-v_\Psi,v-v_\Psi\>^\frac{1}{2}.
\end{equation}
Theorem \ref{lem:minimizingproperty} implies that
\begin{equation}
\<v,v\>=\<v_\Psi,v_\Psi\>+\<v-v_\Psi,v-v_\Psi\>,
\end{equation}
which leads to
\begin{equation}
\<v-v_\Psi,v-v_\Psi\> =\<v,v\>-\<v_\Psi,v_\Psi\>\leq  \<v,v\>,
\end{equation}
which concludes the proof.
\end{proof}

\begin{Remark}\label{rmk:gammaproof}
Observe that Theorem \ref{thm:acchtov} implies that if  $u$ is the solution of the original integro-differential equation \eqref{eqn:scalargeneral} and $\phi_i,\sigma$ are derived from white noise, then
\begin{equation}
\Big\|u-\sum_{i=1}^N \phi_i \big(\int_{\Omega}u(y)\psi_i(y)\,dy \big)\Big\|_{\H(\Omega)} \leq \rho(V_0) \|g\|_{L^2(\Omega)}.
\end{equation}
Similarly,
 if $\phi_i,\sigma$ are derived from  the noise with covariance function $\Lambda$ described in \eqref{eqn:noisyxi}, then
\begin{equation}
\Big\|u-\sum_{i=1}^N \phi_i \big(\int_{\Omega}u(y)\psi_i(y)\,dy \big)\Big\|_{\H(\Omega)} \leq \rho(V_0) \|\L_\Lambda g\|_{L^2(\Omega)}.
\end{equation}
\end{Remark}

\begin{Example}\label{eg:firstone}
If $\L$ and $\B$ correspond to the prototypical example \eqref{eqn:scalar} (Example \ref{eg:8gy2ys}), if $\xi$ is white noise, and if the observable functions are masses of Diracs at points $x_i\in \Omega$ (and $d\leq 3$), then  \cite{OwhadiZhangBerlyand:2014},
\begin{equation}\label{eq:djkjjk3d}
\rho(V_0)\leq C H,
\end{equation}
where $C$ depends only on $\lambda_{\min}(a), \lambda_{\max}(a)$  and
 where $\lambda_{\max}(a):=\sup_{x\in \Omega, l\not=0}l^T a(x) l/|l|^2 $, $\lambda_{\min}(a):=\inf_{x\in \Omega, l\not=0}l^T a(x) l/|l|^2 $ and
 $H$ is the mesh-norm
\begin{equation}\label{eq:meshnorm}
H:=\sup_{x\in \Omega}\min_{i}\|x-x_i\|,
\end{equation}
and
\begin{equation}
\big\|u-\sum_{i=1}^N \phi_i(x)u(x_i)\big\|_{\H^1_0(\Omega)}\leq  C H \big\|\diiv(a\nabla u)\big\|_{L^2(\Omega)}.
\end{equation}
Let us also recall that the proof of \eqref{eq:djkjjk3d} is based on the following Poincar\'{e} inequality (Lemma 3.1 of \cite{OwhadiZhangBerlyand:2014})

\begin{Lemma}\label{lem:murat}{(\cite[Lemma~3.1 ]{OwhadiZhangBerlyand:2014})}
Let $d\leq 3$ and $B_1$ be the open ball of center $0$ and radius $1$. There exists a finite strictly positive constant $C_{\lambda_{\min}(a),\lambda_{\max}(a)}$ such that for all $v\in \H^1(B_1)$
 such that $\diiv(a\nabla v)\in L^2(B_1)$ it holds true that
\begin{equation}\label{eqn:PoincareNarcowichmurat}
\|v-v(0)\|_{L^2(B_1)}^2 \leq C_{\lambda_{\min}(a),\lambda_{\max}(a)}  \Big(\|\nabla v\|_{L^2(B_1)}^2+\big\|\diiv (a\nabla v)\big\|_{L^2(B_1)}^2\Big).
\end{equation}
\end{Lemma}
\begin{proof}
We will recall the proof of this lemma (as presented in \cite[Lemma~3.1 ]{OwhadiZhangBerlyand:2014}) for the sake of completeness. The proof is per absurdum. Note that since $d\leq 3$ the assumptions $v\in \H^1(B_1)$ and $\diiv(a\nabla v)\in L^2(B_1)$ imply the H\"{o}lder continuity of $v$ in $B_1$.
Assume that \eqref{eqn:PoincareNarcowichmurat} does not hold. Then there exists a sequence $v_n$ and a sequence $a_n'$ whose maximum and minimum eigenvalues are uniformly bounded by $\lambda_{\min}(a)$ and $\lambda_{\max}(a)$ (we need to introduce that sequence because we want the constant in \eqref{eqn:PoincareNarcowichmurat} to depend only $d,\lambda_{\min}(a),\lambda_{\max}(a)$)
such that
\begin{equation}\label{eqn:PoincareNarcowichmurat1}
\|v_n-v_n(0)\|_{L^2(B_1)}^2 > n  \Big(\|\nabla v_n\|_{L^2(B_1)}^2+\big\|\diiv (a_n'\nabla v_n)\big\|_{L^2(B_1)}^2\Big)
\end{equation}
Letting $w_n=\frac{v_n-v_n(0)}{\|v_n-v_n(0)\|_{L^2(B_1)}}$ we obtain that $w_n(0)=0$, $\|w_n\|_{L^2(B_1)}=1$ and
\begin{equation}\label{eqn:PoincareNarcowichmurat2}
\|\nabla w_n\|_{L^2(B_1)}^2+\big\|\diiv (a_n'\nabla w_n)\big\|_{L^2(B_1)}^2 < \frac{1}{n}
\end{equation}
Since
\begin{equation}\label{eqn:PoincareNarcowichmurat3}
\| w_n\|_{\H^1(B_1)}< 1+\frac{1}{n}\leq 2
\end{equation}
it follows that there exists a subsequence $w_{n_j}$ and a $w\in \H^1(B_1)$ such that $w_{n_j}\rightharpoonup w$ weakly in $\H^1(B_1)$ and
$\nabla w_{n_j}\rightharpoonup \nabla w$ weakly in $L^2(B_1)$.  Using $\|\nabla w_n\|_{L^2(B_1)}\leq 1/n$ we deduce that $\nabla w=0$ which implies that $w$ is a constant in $B_1$. Since by the Rellich–-Kondrachov theorem the embedding $\H^1(B_1)\subset L^2(B_1)$ is compact it follows from \eqref{eqn:PoincareNarcowichmurat3} that $w_{n_j}\rightarrow w$ strongly in $L^2(B_1)$ which (using $\|w_n\|_{L^2(B_1)}=1$) implies that $\|w\|_{L^2(B_1)}=1$.
Now \eqref{eqn:PoincareNarcowichmurat3}  together with the fact that $\big\|\diiv (a_n'\nabla w_n)\big\|_{L^2(B_1)}^2$ is uniformly bounded and that $d\leq 3$ implies that $w_n$ is uniformly H\"{o}lder continuous on $B(0,\frac{1}{2})$ (see for instance \cite{Stampaccia:1964}). This implies that $w$ is continuous in $B(0,\frac{1}{2})$ and that $w(0)=0$. This contradicts the fact that $w$ is a constant in $B_1$ with $\|w\|_{L^2(B_1)}=1$.
\end{proof}

\end{Example}

\begin{Example}
If $\L$ and $\B$ correspond to the prototypical example \eqref{eqn:scalar} (Example \ref{eg:8gy2ys}), if $\xi$ is white noise, and if the observable functions are indicator functions of Vorono\"{i} cells around points in  $x_i\in \Omega$ or of tetrahedra of a regular tessellation of the points $x_i \in \Omega$ then
\eqref{eq:djkjjk3d} remains valid as a simple consequence of localized Poincar\'{e} inequalities. Indeed for $v\in V_0$, writing $C_i$ the Vorono\"{i} cells at
the points  $x_i\in \Omega$, we have (assuming $\Omega$ is the union of those Vorono\"{i} cells)

\begin{equation}
\begin{split}
\|v\big\|_{\H^1_0(\Omega)}^2 = \int_{\Omega} v(x)\big(-\diiv(a(x)\nabla v(x))\big)\,dx\leq \|v\|_{L^2(\Omega)} \big\|\diiv(a\nabla v)\big\|_{L^2(\Omega)},
\end{split}
\end{equation}
and we conclude by applying Poincar\'{e}'s inequality to the $L^2$-norm of $v$ within each cell $C_i$, i.e.
\begin{equation}
\begin{split}
 \|v\|_{L^2(\Omega)}^2 =\sum_{i}  \|v\|_{L^2(C_i)}^2\leq C H^2  \sum_{i}  \|\nabla v\|_{L^2(C_i)}^2=  C H^2 \|\nabla v\|_{L^2(\Omega)}^2.
\end{split}
\end{equation}
\end{Example}

We will give the last example as a theorem.
\begin{Theorem}
Let $\L$ and $\B$ be as in the prototypical example \eqref{eqn:scalar} (Example \ref{eg:8gy2ys}) and let $\xi$ be white noise. Let $\psi_1,\ldots,\psi_N$ be linearly independent generalized probability densities on $\Omega$ with (possibly overlapping) support $\operatorname{support}(\psi_i)$. Define
\begin{equation}
H:=\sup_{x\in \Omega}\min_i \sup_{y\in \operatorname{support}(\psi_i)}\|x-y\|.
\end{equation}
Then, it holds true that
\begin{equation}\label{eq:djkjjk3dsee}
\rho(V_0)\leq C H,
\end{equation}
where $C$ depends only on $\lambda_{\min}(a)$ and  $\lambda_{\max}(a)$. Henceforth, for $u \in V$
\begin{equation}
\big\|u-\sum_{i=1}^N \phi_i(x)\int_{\Omega}u(y)\psi_i(y)\,dy\big\|_{\H^1_0(\Omega)}\leq  C H \big\|\diiv(a\nabla u)\big\|_{L^2(\Omega)}.
\end{equation}
\end{Theorem}
\begin{Remark}
Observe that if  
 for all $i$ the support of $\psi_i$ is contained in a ball of center $x_i$ and radius $H'$, then
\begin{equation}
H\leq H'+\sup_{x\in \Omega}\min_{i}\|x-x_i\|,
\end{equation}
in particular if the points $x_i$ have mesh norm $H''$ (see \eqref{eq:meshnorm}) then $H\leq H'+H''$.
\end{Remark}
\begin{proof}
The proof of \eqref{eq:djkjjk3dsee} is simply based on the observation that if $v\in V_0$ then (since $\int_{\Omega}v(x)\psi_i(x)\,dx=0$) there exists $N$ points $y_1,\ldots,y_N$ such that $v(y_i)=0$ and the mesh norm of those points is bounded by $H$. Therefore we can apply the result of Example \ref{eg:firstone}.
\end{proof}

\section{Pseudo-algorithm}
A simple pseudo-algorithmic description of the proposed framework for the numerical homogenization of \eqref{eqn:scalargeneral} is as follows: 
\begin{enumerate}
\item Select $N$ linearly independent (measurement) functions $\psi_1,\ldots,\psi_N$ in $L^2(\Omega)$.
 \item Let $\xi$ in \eqref{eqn:noisy} be a Gaussian field of mean $0$ and covariance function $\Lambda(x,y)$ (assumed to be non-degenerate, i.e. such that there exists an inverse covariance function $\Lambda^{-1}(x,y)$ with $\int_{\Omega^2} \Lambda(x,y) \Lambda^{-1}(y,z)\,dy=\delta(x-z)$).
\item  The basis functions $\phi_1,\ldots,\phi_N$ for the numerical homogenization of \eqref{eqn:scalargeneral} are identified as (writing $u$ the solution of \eqref{eqn:noisy} and $\delta_{i,j}=1$ if $i=j$ and $\delta_{i,j}=0$ if $i\not=j$) the deterministic functions
\begin{equation}\label{eqgjshjshg}
\phi_i(x)=\E\big[u(x)\big|\int_{\Omega}u(x)\psi_j(x)\,dx=\delta_{i,j}\text{ for }j=1,\ldots,N\big].
\end{equation}
\item Each $\phi_i$ can also be identified as the unique minimizer of 
\begin{equation}\label{eqn:convexoptsimplsse}
\begin{cases}
\text{Minimize }  \int_{\Omega^2}(\L u(x)) \Lambda^{-1}(x,y) (\L u(y))\,dx\,dy  \\
\text{Subject to } \phi\in \H(\Omega)\text{ and }\int_{\Omega}\phi(x)\psi_j(x)\,dx=\delta_{i,j}\text{ for }j=1,\ldots,N
\end{cases}
\end{equation}
\item Under appropriate choice of the measurement functions $\psi_i$ and the covariance function $\Lambda(x,y)$, the basis functions $\phi_i$ can be computed by localizing the optimization problems \eqref{eqn:convexoptsimplsse} to subdomains of $\Omega$.
\end{enumerate}

\section{Statistical Decision Theory and Practical Applications}
Another motivation for exploring Bayesian approximations of the solution space, lies in the decision theory/game theory approach to numerical homogenization.
In this approach one looks at the numerical homogenization problem \eqref{eqn:scalar} as a repeated game where player B chooses a function $\theta$ of the linear measurements (data) $\int_{\Omega}u(x)\psi_1(x)\,dx,$ $\ldots,$ $\int_{\Omega}u(x)\psi_N(x)\,dx$ and player A chooses a source term $g$ in the unit ball of $L^2(\Omega)$. These two choices combine and form an error term
\begin{equation}\label{eq:loss}
\mathcal{E}(\theta,g)=\Big\|u-\theta\big(\int_{\Omega}u(x)\psi_1(x)\,dx,\ldots,\int_{\Omega}u(x)\psi_N(x)\,dx\big)\Big\|_{L^2(\Omega)}.
\end{equation}
Player's B objective is to minimize the error \eqref{eq:loss} while player's A objective is to maximize it. A surprising result stemming from a generalization \cite{OwhadiMultigrid:2015}  of Wald's Decision Theory \cite{Wald:1945} and Von Neumann's Game Theory \cite{VonNeumann:1944} is that, although such games are deterministic, under weak regularity conditions, the optimal strategy for player $A$ is to play at random by placing an optimal probability distribution $\pi_A$  on the set of candidates for $g$ and, similarly, the best strategy for player $B$ is to assume that player A is playing at random  and to use a function $\theta$ living in the Bayesian class (obtained by placing a prior $\pi_B$  on the set of candidates for $g$ and conditioning with respect to the measurements $\int_{\Omega}u(x)\psi_i(x)\,dx$).

Although the estimator employed by player B may be called Bayesian, the game described here is not (i.e. the choice of player A might be distinct from that of player B) and player B must solve a min max optimization problem over $\pi_A$ and $\pi_B$ to identify an optimal prior distribution for the Bayesian estimator (a careful choice of the prior also appears to be important due to the possible high sensitivity of posterior distributions  \cite{owhadiBayesiansirev2013, OSS:2013, OwhadiScovel:2013, OwhadiScovelQR:2014}).

We refer to \cite{OwhadiMultigrid:2015} for (1) the complete description of the generalization of the Bayesian framework described here to the 
 decision theory/information game formulation (described above) (2)  practical (including numerical) applications of that generalized framework
to the problems of finding numerical homogenization bases and fast solvers for 
 \eqref{eqn:scalar}. In that generalization, optimal numerical homogenization bases functions are obtained by  selecting the prior distribution of $\xi$ (in \eqref{eqn:scalarstochastic}) to be that of a Gaussian field with mean zero and covariance function the operator \eqref{eqn:scalar} (i.e. such that for $f\in H^1_0(\Omega)$, $\int_{\Omega}f(x)\xi(x)\,dx$ is a Gaussian random variable of mean zero and variance $\int_{\Omega}(\nabla f(x))^T a(x) \nabla f(x)\,dx$). In particular  \cite{OwhadiMultigrid:2015} shows how  the identification of an optimal distribution for $\xi$ (in the Gaussian class) leads to the (automated) discovery  of multigrid and multiresolution solvers for PDEs with rough coefficients.

\paragraph{Acknowledgements.}
The author gratefully acknowledges this work supported by  the Air Force Office of Scientific Research under
 Award Number  FA9550-12-1-0389 (Scientific Computation of Optimal Statistical Estimators) and the U.S. Department of Energy Office of Science, Office of Advanced Scientific Computing Research, through the Exascale Co-Design Center for Materials in Extreme Environments (ExMatEx, LANL Contract No
DE-AC52-06NA25396, Caltech Subcontract Number 273448).
 The author also thanks Dongbin Xiu, Lei Zhang and Guillaume Bal for stimulating discussions and Leonid Berlyand for comments on the manuscript.
The author also thanks two anonymous referees for valuable comments and suggestions.

\bibliographystyle{plain}
\bibliography{RPS}

\end{document}